\documentclass[a4paper,10pt]{article}

\usepackage[dvipdfm,%
 bookmarks=true,%
 bookmarksnumbered=true,%
 colorlinks=true,%
 pdftitle={},%
 pdfauthor={Yokoyama},%
 pdfsubject={},%
 pdfkeywords={}]{hyperref}

\makeatletter
\def\blfootnote{\gdef\@thefnmark{}\@footnotetext}
\makeatother

\usepackage{soul}

\usepackage{latexsym}
\usepackage{amsmath}
\usepackage{amscd}
\usepackage{amssymb}
\usepackage{authblk}
\usepackage[bbgreekl]{mathbbol}

\usepackage{url}
\usepackage{amsthm}

\def\RCAo{\mathsf{RCA_0}}

\def\WKLo{\mathsf{WKL_0}}
\def\ACAo{\mathsf{ACA_0}}

\def\WKL{\mathsf{WKL_0}}

\def\E{\exists}
\def\A{\forall}
\def\N{\mathbb{N}}

\def\lh{\mathrm{lh}}

\def\rest{{\upharpoonright}}

\def\P2{\Pi^1_2}

\def\PHt{\mathrm{PH}^2_2}
\def\RT{\mathrm{RT}}

\def\BII{\mathrm{B}\Sigma^0_2}

\def\II{\mathrm{I}\Sigma^0_1}
\def\III{\mathrm{I}\Sigma^0_2}

\newcounter{menum}
{\begin{enumerate}%
\setcounter{enumi}{#1}}%
{\setcounter{menum}{\value{enumi}}\end{enumerate}}

\newtheorem{thm}{Theorem}[section]

\newtheorem*{theorem*}{Theorem}

\newtheorem*{claim*}{Claim}
\newtheorem{prop}[thm]{Proposition}
\newtheorem{lem}[thm]{Lemma}
\newtheorem{cor}[thm]{Corollary}

\theoremstyle{definition}
\newtheorem{defi}{Definition}[section]

\newtheorem{question}[defi]{Question}

\usepackage{xcolor}

\definecolor{lightred}{rgb}{1,.60,.60}

\begin{document}
\newpage

\title{Very weak fragments of weak K\H{o}nig's lemma}

\author[1]{Stephen G.~Simpson}
\author[2]{Keita Yokoyama}

\affil[1]{\small \sf{sgslogic@gmail.com}}
\affil[2]{\small \sf{y-keita@jaist.ac.jp}}

\date{}

\blfootnote{%
The work of the second author is partially supported by
 JSPS KAKENHI grant number 16K17640,
 JSPS fellowship for research abroad,
 JSPS-NUS Bilateral Joint Research Projects J150000618 (PI's: K.~Tanaka, C.~T.~Chong),
 and JSPS Core-to-Core Program (A.~Advanced Research Networks).

The major part of this work was done when the second author visited Vanderbilt university in July 2016.
 }

\maketitle
\def\Bexp{\mathrm{B}\Sigma_{1}+\mathrm{exp}}
\newcommand\RF{\mathrm{RF}}
\newcommand\tpl{\mathrm{tpl}}
\newcommand\col{\mathrm{col}}
\newcommand\fin{\mathrm{fin}}
\newcommand\Log{\mathrm{Log}}
\newcommand\Ct{\mathrm{Const}}
\newcommand\It{\mathrm{It}}
\renewcommand\PHt{\widetilde{\mathrm{PH}}{}}
\newcommand\BME{\mathrm{BME}_{*}}
\newcommand\HT{\mathrm{HT}}
\newcommand\Fin{\mathrm{Fin}}
\newcommand\FinHT{\mathrm{FinHT}}
\newcommand\wFinHT{\mathrm{wFinHT}}
\newcommand\FS{\mathrm{FS}}
\newcommand\LL{\mathsf{L}}
\newcommand\GPg{\mathrm{GP}}
\newcommand\GP{\mathrm{GP}^{2}_{2}}
\newcommand\FGPg{\mathrm{FGP}}
\newcommand\FGP{\mathrm{FGP}^{2}_{2}}
\newcommand\SGP{\mathrm{SGP}^{2}_{2}}
\newcommand\Con{\mathrm{Con}}
\newcommand\WF{\mathrm{WF}}
\newcommand\bb{\mathbf{b}}
\renewcommand\WKL{\mathrm{WKL}}
\newcommand\KL{\mathrm{KL}}
\newcommand\sWKL{\Sigma^{0}_{1}\text{-}\mathrm{WKL}}
\newcommand\ext{\mathrm{ext}}
\newcommand\VS{\mathrm{VSMALL}}

%

\section{Introduction}\label{sec:preliminary}
It is well-known that any finite $\Pi^{0}_{1}$-class of $2^{\N}$ has a computable real.
Then, how can we understand this in the context of reverse mathematics?
We need to formalize the statement carefully.
Within $\RCAo$, the assertion ``every infinite binary tree which has at most one path has a path'' is already equivalent to $\WKL$ since the negation of $\WKL$ implies the existence of an infinite binary tree with no path.
In this note, we will consider the following very weak versions of K\"onig's lemma.
\begin{defi}\label{defi-WWWKL}
\begin{enumerate}
 \item $\WKL(pf\text{-}bd)$: an infinite binary tree $T\subseteq 2^{<\N}$ has a path if there exists $c\in\N$ such that for any prefix-free set $P\subseteq T$, $|P|\le c$.
 \item $\WKL(w\text{-}bd)$: an infinite binary tree $T\subseteq 2^{<\N}$ has a path if there exists $c\in\N$ such that for any $n\in\N$, $|T^{=n}|\le c$, where $T^{=n}=\{\sigma\in T\mid \lh(\sigma)=n\}$.
 \item $\WKL(ext\text{-}bd)$: an infinite binary tree $T\subseteq 2^{<\N}$ has a path if there exists $c\in\N$ such that for any $n\in\N$, $|T_{\ext}^{=n}|\le c$, where $T_{\ext}^{=n}=\{\sigma\in T\mid \lh(\sigma)=n\wedge \sigma\mbox{ is extendible}\}$.
 \item $\sWKL(pf\text{-}bd)$: an $\Sigma^{0}_{1}$ infinite binary tree $T\subseteq 2^{<\N}$ has a path if there exists $c\in\N$ such that for any prefix-free set $P\subseteq T$, $|P|\le c$.
 Here, $\Sigma^{0}_{1}$ tree is described by an enumeration of binary strings $\{\sigma_{i}\}_{i\in\N}$ such that $T=\{\sigma_{i}\mid i\in\N\}$ forms a tree.
 \item $\sWKL(w\text{-}bd)$: an $\Sigma^{0}_{1}$ infinite binary tree $T\subseteq 2^{<\N}$ has a path if there exists $c\in\N$ such that for any $n\in\N$, $|T^{=n}|\le c$.
 This statement is known as Chaitin's lemma.
 \item $\sWKL(ext\text{-}bd)$: an $\Sigma^{0}_{1}$ infinite binary tree $T\subseteq 2^{<\N}$ has a path if there exists $c\in\N$ such that for any $n\in\N$, $|T_{\ext}^{=n}|\le c$.
 \item $\KL(pf\text{-}bd)$: an infinite binary tree $T\subseteq \N^{<\N}$ has a path if there exists $c\in\N$ such that for any prefix-free set $P\subseteq T$, $|P|\le c$.
 \item $\KL(w\text{-}bd)$: an infinite binary tree $T\subseteq \N^{<\N}$ has a path if there exists $c\in\N$ such that for any $n\in\N$, $|T^{=n}|\le c$.
 \item $\KL(ext\text{-}bd)$: an finitely-branching infinite binary tree $T\subseteq \N^{<\N}$ has a path if there exists $c\in\N$ such that for any $n\in\N$, $|T_{\ext}^{=n}|\le c$.
\end{enumerate} 
\end{defi}
Trivially, 3 $\to$ 2 $\to$ 1, 6 $\to$ 5 $\to$ 4 and 9 $\to$ 8 $\to$ 7 hold.
One may observe that any $\Sigma^{0}_{1}$ tree is easily interpreted as a tree in $\N^{<\N}$.
Indeed, for a given a $\Sigma^{0}_{1}$ tree $T=\{\sigma_{i}\}_{i\in\N}$, set $\hat{T}\subseteq \N^{<\N}$ as $\tau\in \hat{T}\leftrightarrow \A i\le j<\lh(\tau)(\lh(\sigma_{\tau(i)})=i\wedge\sigma_{\tau(i)}\subseteq\sigma_{\tau(j)})$,
then, $T$ and $\hat{T}$ is isomorphic and a path of $\hat{T}$ computes a path of $T$.
Thus, we have 7 $\to$ 4 $\to$ 1, 8 $\to$ 5 $\to$ 2 and 9 $\to$ 6 $\to$ 3.

It is well-known that any $\Sigma^{0}_{1}$-definable set can be described as a unique path of a $\Sigma^{0}_{1}$ tree.
\begin{prop}[$\RCAo$, folklore]
For any $\Sigma^{0}_{1}$-definable set $A\subseteq \N$, there exists a $\Sigma^{0}_{1}$ tree $T$ such that $|T_{\ext}^{=n}|=1$ for any $n\in\N$ and $A$ is a path of $T$.
\end{prop}
\begin{proof}
Write $n\in A\leftrightarrow \E m\theta(m,n)$ for some $\Sigma^{0}_{0}$ formula $\theta$.
Then, define a $\Sigma^{0}_{1}$ tree $T$ as $\sigma\in T\leftrightarrow \E m\in\N\A i<\lh(\sigma)(\sigma(i)=1\leftrightarrow \E m'<m\,\theta(m',i))$.
It is easy to check that this $T$ is the desired.
\end{proof}

\begin{cor}
The following are equivalent over $\RCAo$.
\begin{enumerate}
 \item $\ACAo$.
 \item $\sWKL(ext\text{-}bd)$.
 \item $\KL(ext\text{-}bd)$.
\end{enumerate}
\end{cor}
Thus, they are not computably weak any more.
On the other hand, all other statements in Definition~\ref{defi-WWWKL} is computably true, in other words, it is true in any $\omega$-model of $\RCAo$.
However, we will see that most of them are not provable within $\RCAo$, thus they require some non-trivial induction.

Throughout this note, we will use the following well-known fact.
\begin{lem}[$\RCAo$]\label{lem:neg-ISig2}
If $\III$ fails, then there exists a set $X$ and a $\Pi^{0,X}_{1}$-set $A$ such that $A$ is unbounded and $|A|\le c$ for some $c\in\N$.
(Here, $A$ is said to be unbounded if $\A n\in\N \E m\ge n\, m\in A$, and
$|A|\le c$ means that for any finite set $F\subseteq A$ (coded by a natural number), $|F|\le c$.)
\end{lem}

\section{Computably true fragments of $\KL$ and induction}

\begin{thm}
The following are equivalent over $\RCAo$.
\begin{enumerate}
 \item $\BII$.
 \item $\sWKL(pf\text{-}bd)$.
 \item $\KL(pf\text{-}bd)$.
\end{enumerate} 
\end{thm}
\begin{proof}
We first show 1 $\to$ 3.
Let $b\in\N$, and let $T\subseteq \N^{<\N}$ be an infinite tree such that for any prefix-free $P\subseteq T$, $|P|\le b$.
By $\Sigma^{0}_{1}$-induction, $b_{0}=\max \{b'\le b\mid \E P\subseteq T(P$ is prefix-free and $|P|=b')\}$ exists.
Take prefix-free $P_{0}\subseteq T$ so that $|P_{0}|=b_{0}$.
Then, every element in $T$ is compatible with some member in $P_{0}$ by the maximality.
Put $\hat{T}=\{\sigma\in T\mid \E \tau\in P_{0}\, \sigma\supseteq \tau\}$.
Then, $\hat{T}$ is infinite.
For any $\sigma\in \hat{T}$, there exists at most one immediate extension.
(Presume that $\sigma_{1}$ and $\sigma_{2}$ are immediate extensions of $\sigma$ and $\sigma\supset \tau\in P_{0}$, then $P_{0}\setminus\{\tau\}\cup\{\sigma_{1},\sigma_{2}\}$ is a prefix-free set, which contradicts the maximality of $P_{0}$.)
By $\BII$, there exists $\tau_{0}\in P_{0}$ such that there exist infinitely many extensions of $\tau_{0}$ in $\hat{T}$.
Thus, any extension of $\tau_{0}$ has exactly one immediate extension.
Hence one can compute a path extending $\tau_{0}$.

3 $\to$ 2 is explained in Section~\ref{sec:preliminary}.

Finally we show $\neg$1 $\to$ $\neg$2.
Since $\BII$ is equivalent to $\RT^{1}$, we assume that there exists a function $h:\N\to b$ for some $b\in\N$ such that $h^{-1}(x)$ is finite for all $x<b$.
Put $T=\{1^{x}0^{y}\in 2^{<\N}\mid \E x<b\, \E y'\ge y\, h(y')=x\}$.
Then, $T$ is a $\Sigma^{0}_{1}$ tree, and it is infinite since $1^{h(n)}0^{n}\in T$ for any $n\in\N$.
Moreover, any prefix-free subset of $T$ is bounded by $b$.
Since $h^{-1}(x)$ is finite for all $x<b$, there is no infinite path of $T$, thus we have $\neg$2.
\end{proof}

\begin{lem}\label{lem:ISig2>ext-bd}
$\RCAo+\III$ proves $\WKL(ext\text{-}bd)$.
\end{lem}
\begin{proof}
Let $T\subseteq 2^{<\N}$ be an infinite tree and let $b\in \N$ be a bound for $T_{\ext}^{=n}$.
By $\III$ there exists $b_{0}\le b$ such that $b_{0}=\min\{a\le b\mid \A n(|T^{=n}_{\ext}|\le a)\}$.
Take $n\in\N$ such that $|T^{=n}_{\ext}|=b_{0}$, and let $\sigma\in T^{=n}_{\ext}$.
Then, any extendible extension of $\sigma$ has exactly one immediate extendible extension.
(More formally, for any $\sigma'\supseteq \sigma$, there exist $n\in\N$ and $i< 2$ such that for any $\tau\in T^{=n}$, $\tau\supseteq \sigma'\to \tau\supseteq \sigma'{}^{\frown}i$.)
Thus, one can compute a path extending $\sigma$.
\end{proof}

\begin{thm}
The following are equivalent over $\RCAo$.
\begin{enumerate}
 \item $\III$.
 \item $\sWKL(w\text{-}bd)$.
 \item $\KL(w\text{-}bd)$.
\end{enumerate} 
\end{thm}
\begin{proof}
We first show 1 $\to$ 3.
Let $T\subseteq \N^{<\N}$ be an infinite tree and let $b\in \N$ be a bound for $T^{=n}$.
By $\III$ there exists $b_{0}\le b$ such that $b_{0}=\max\{a\le b\mid \A n\E m(|T^{=m}_{\ext}|\ge a)\}$.
Then, the $\Sigma^{0}_{1}$ set $X_{0}=\{m\in\N\mid |T^{=m}_{\ext}|\ge b_{0}\}$ is infinite, and hence there exists an infinite set $X\subseteq X_{0}$.
By maximality, $|T^{=m}_{\ext}|= b_{0}$ for all but finite $m\in X$, so we may assume that $|T^{=m}_{\ext}|= b_{0}$ holds for all $m\in X$.
Write $X=\{m_{0}<m_{1}<\dots\}$, and put $\sigma_{i,j}$ be the $j$-th left-most element of $T^{=m_{i}}$.
Note that one can compute the double sequence $\{\{\sigma_{i,j}\mid j<b_{0}\}\mid i\in\N\}$ from $T$ and $X$.
Now, define an infinite tree $S\subseteq b_{0}^{<\N}$ as $\tau\in S\leftrightarrow \A i\le j<\lh(\tau)(\sigma_{i,\tau(i)}\subseteq\sigma_{j,\tau(j)})$.
Then, $S$ has a path by $\WKL(w\text{-}bd)$, which is provable from $\III$ by Lemma~\ref{lem:ISig2>ext-bd}.
One can easily retrieve a path of $T$ from a path of $S$.

3 $\to$ 2 is explained in Section~\ref{sec:preliminary}.

Finally we show $\neg$1 $\to$ $\neg$2.
By Lemma~\ref{lem:neg-ISig2}, there exists a set $X$ and a $\Pi^{0,X}_{1}$-set $A$ such that $A$ is unbounded and $|A|\le c$ for some $c\in\N$.
Note that $A$ cannot exist as a set since the cardinality of an unbounded set won't be bounded within $\RCAo$.
Write $n\in A\leftrightarrow \A m\theta(m,n,X)$ where $\theta$ is a $\Sigma^{0}_{0}$-formula.
Then, define a $\Sigma^{0}_{1}$ tree $T$ as 
\begin{align*}
\sigma\in T\leftrightarrow \E m>\lh(\sigma)( & \A i<\lh(\sigma)(\sigma(i)=1\leftrightarrow \A m'<m\,\theta(m',i,X))\\
& \wedge |\{i<\lh(\sigma)\mid \A m'<m\,\theta(m',i,X)\}|\le c).
\end{align*}
Then, $T$ is infinite since $A\rest n:=\{x\in A\mid x<n\}\in T$ for any $n\in\N$.
($A\rest n$ always exists by bounded $\Sigma^{0}_{1}$-comprehension which is available within $\RCAo$.)
By the definition of $T$, for any $\sigma,\tau\in T$ such that $\lh(\sigma)=\lh(\tau)$, if $\sigma(i)<\tau(i)$ for some $i<\lh(\sigma)$, then $\sigma(i)\le \tau(i)$ for all $i<\lh(\sigma)$.
Thus, there are at most $c$-many elements in $T^{=n}$ for any $n\in\N$.
A path of $T$ should be identical with $A$, so $T$ cannot have a path.
\end{proof}

\section{Computably true fragments of $\WKL$ and induction}
The situation is more complicated when we consider weak fragments of $\WKL$ since they are all provable within $\WKLo$, which is $\Pi^{1}_{1}$-conservative over $\RCAo$.
Thus, they never imply pure induction axioms.
Still, they may require some induction when $\WKL$ fails badly.

First, we see that a bound for prefix-free subsets is enough to find a recursive path even within $\RCAo$.
\begin{prop}
$\WKL(pf\text{-}bd)$ is provable within $\RCAo$.
\end{prop}
\begin{proof}
Let $T\subseteq 2^{<\N}$ be an infinite tree and let $c\in\N$ such that for any prefix-free set $P\subseteq T$, $|P|\le c$.
By $\II$ there exists $c_{0}\le c$ and a prefix-free set $P_{0}\subseteq T$ such that $|P_{0}|=c_{0}$ and there is no prefix-free set $P'\subseteq T$ with $|P'|>c_{0}$.
Then, there exists an extendible node $\sigma\in T$ such that $\sigma\in P_{0}$.
By the maximality of $c_{0}$ and $P_{0}$, there is no incomparable $\tau,\tau'\in T$ extending $\sigma$, hence one may easily computes a path of $T$ extending $\sigma$. 
\end{proof}

Next we focus on $\WKL(w\text{-}bd)$ and $\WKL(ext\text{-}bd)$. Indeed, they are equivalent and strictly in between $\RCAo$ and $\WKLo$.
\begin{thm}
$\WKL(w\text{-}bd)$ and $\WKL(ext\text{-}bd)$ are equivalent over $\RCAo$.
\end{thm}
\begin{proof}
Let $T\subseteq 2^{<\N}$ be an infinite tree such that $\A n\in\N|T^{=n}_{\ext}|\le c$ for some $c\in\N$.
We will construct a tree $T'\subseteq 3^{<\N}$ with $|T'^{=n}|\le c$ such that any path of $T'$ computes a path of $T$.
Put $T_{s}=\{\sigma\in T\mid \E \tau\in T^{=\lh(\sigma)+s}\, \sigma\subseteq\tau\}$.
We recursively define $\{s_{n}\}_{n\in\N}$ as $s_{n}=\min \{s>s_{n-1}\mid |T_{s}^{=n}|\le c\}$.
Such $s_{n}$ always exists since $|T^{=n}_{\ext}|\le c$.
Now $T'\subseteq 3^{<\N}$ is defined as follows:
\begin{align*}
\tau\in T'\leftrightarrow \E \sigma\in T^{\le s}\A s<\lh(\tau)(\tau(s)=\sigma(n)\mbox{ if $\E n\le s\, s=s_{n}$}\wedge \tau(s)=2\mbox{ if $\A n\le s\, s\neq s_{n}$}). 
\end{align*}
Now, any string $\tau\in T'$ is of the form $\tau=j_{0}^{\frown}2^{i_{0}}{}^{\frown}j_{1}^{\frown}2^{i_{1}}{}^{\frown}\dots^{\frown}j_{n-1}^{\frown}2^{i_{n-1}}$ such that $\langle j_{0},j_{1},\dots,j_{n-1} \rangle\in T_{s}^{=n}$ for some $s\ge s_{n}$.
Thus, $|T'^{=n}|\le c$.
For a given a path of $T'$, one can easily compute a path of $T$ by removing all 2's.
\end{proof}

\begin{thm}[$\RCAo$]\label{thm:small-tree}
If $\III$ fails, there exists a tree $T\in S$ such that $\A n\in\N |T_{\ext}^{=n}|\le c$ for some $c\in\N$ and $[T]$ does not contain any $T$-recursive elements.
\end{thm}
\begin{proof}
We argue within $(M,S)\models \RCAo+\neg\III$.
Then, there exists a set $X\in S$ and an $X$-r.e. set $A$ such that $\N\setminus A$ is unbounded and $|\N\setminus A|\le c$ for some $c\in\N$.
Such $A$ cannot be a member of $S$ since there is no unbounded set whose cardinality is bounded within $\RCAo$.
Thus, $A$ is not $X$-recursive in $(M,S)$.
By formalizing the construction of \cite[Lemma 8.2]{Simpson-Pi01class}, there exists a pair of disjoint $X$-r.e.~sets $B_{0}$ and $B_{1}$ which splits $A$ such that their separating class does not contain any $X$-recursive elements.
Take an $X$-recursive tree $T$ such that $[T]$ is the separating class of $B_{0}$ and $B_{1}$.
If $\sigma,\tau\in T_{\ext}^{=n}$, $\sigma(i)=\tau(i)$ for all $i\in A$.
Thus, $|T_{\ext}^{=n}|\le 2^{c}$.
\end{proof}

\begin{cor}
$\RCAo$ does not imply $\WKL(ext\text{-}bd)$.
\end{cor}

\begin{cor}
The following are equivalent over $\RCAo$.
\begin{enumerate}
 \item $\III$.
 \item recursive-$\WKL(w\text{-}bd)$: an infinite binary tree $T\subseteq 2^{<\N}$ has a $T$-recursive path if there exists $c\in\N$ such that for any $n\in\N$, $|T^{=n}|\le c$.
 \item recursive-$\WKL(ext\text{-}bd)$: an infinite binary tree $T\subseteq 2^{<\N}$ has a $T$-recursive path if there exists $c\in\N$ such that for any $n\in\N$, $|T_{\ext}^{=n}|\le c$.
\end{enumerate}
\end{cor}
\begin{cor}
Over $\RCAo$, $\WKL(w\text{-}bd)$ plus the assertion ``there exists a set $X$ such that for any set $Y$, $Y\le_{T}X$'' implies $\III$.
\end{cor}

%

\begin{defi}[Very smallness, Binns/Kjos-Hanssen\cite{BK2009}]
$\VS$ asserts the following: an infinite binary tree $T\subseteq 2^{<\N}$ has a path if for any function $f:\N\to\N$, there exists $n\in\N$ such that for any $m\ge n$, $|T^{=f(m)}_{\ext}|<m$.
\end{defi}


\begin{prop}
Over $\RCAo$, the disjunction of $\III$ and $\VS$ implies $\WKL(w\text{-}bd)$.
\end{prop}
\begin{proof}
By Lemma~\ref{lem:ISig2>ext-bd} and the definition of $\VS$.
\end{proof}

\begin{question}
Is $\WKL(w\text{-}bd)$ strictly weaker than $\III\vee\VS$ over $\RCAo$?
\end{question}

\bibliographystyle{plain}
\bibliography{bib}

\end{document}